\documentclass[10pt]{amsart}
\usepackage[cp1251]{inputenc}
\usepackage[english,russian]{babel}
\usepackage{amsmath}
\usepackage{amssymb}
\usepackage{amsfonts}
\usepackage{url}
\usepackage[dvips]{graphicx} 
\graphicspath{{Risunki/}}
\usepackage{tikz}

\def\udcs{519.174.7} 
\def\mscs{05C50} 
\newtheorem{lemma}{Lemma}
\newtheorem{theorem}{Theorem}

\newtheorem{corollary}{Corollary}

\def\logo{{\bf\huge S\raisebox{0.2ex}{\hspace{0.55ex}\raisebox{0.05ex}e\hspace{-1.65ex}$\bigcirc$}MR}}

\def\semrtop
     {
  \vbox{
     \noindent\logo
     \hspace{80mm}\raisebox{1ex}{ISSN 1813--3304 }

     \vspace{5mm}

     \begin{center}
     {\huge СИБИРСКИЕ \ ЭЛЕКТРОННЫЕ} \\[2mm]
     {\huge МАТЕМАТИЧЕСКИЕ ИЗВЕСТИЯ} \\[2mm]
     {\large Siberian Electronic Mathematical Reports} \\[1mm]
     {\LARGE\tt{http://semr.math.nsc.ru}}\\[0.5mm]
     \end{center}
     \vspace{-3mm}
     \noindent
     \begin{tabular}{c}
     \hphantom{aaaaaaaaaaaaaaaaaaaaaaaaaaaaaaaaaaaaaaaaaaaaaaaaaaaaaaaaaaaaaaaaaaaaaa} \\
     \hline\hline
     \end{tabular}

     \vspace{1mm}
     {\flushleft\it Том 16, стр. 144--144 (2019) \hspace{65mm}{\rm\small УДК \udcs}}
     \newline
        {\rm\small DOI~10.33048/semi.2019.16.xxx}\hphantom{aaaaaaaaaaaaaaaaaaaaaaaaaaaaaaaaaaaa}{\rm\small MSC\ \ \mscs }
  }
}

\begin{document}
\renewcommand{\refname}{References}
\renewcommand{\proofname}{Proof}

\thispagestyle{empty}

\title[On perfect colorings of infinite multipath graphs]{On perfect colorings of infinite multipath graphs}
\author{{M. A. LISITSYNA, S. V. AVGUSTINOVICH, O. G. PARSHINA}}%

\address{Mariya Aleksandrovna Lisitsyna  
\newline\hphantom{iii} Budyonny Military Academy of the Signal Corps,
\newline\hphantom{iii} pr. Tikhoretsky, 3,
\newline\hphantom{iii} 194064, St Petersburg, Russia}%
\email{lisitsyna.mariya.mathematician@gmail.com}%

\address{Sergey Vladimirovich Avgustinovich  
\newline\hphantom{iii} Sobolev Institute of Mathematics,
\newline\hphantom{iii} pr. Koptyuga, 4,
\newline\hphantom{iii} 630090, Novosibirsk, Russia}%
\email{avgust@math.nsc.ru}%

\address{Olga Gennad'evna Parshina 
\newline\hphantom{iii} Czech Technical University in Prague,
\newline\hphantom{iii} Trojanova str., 13
\newline\hphantom{iii} 120 00, Prague, Czech Republic}%
\email{parolja@gmail.com}

\thanks{\sc Lisitsyna, M.A., Avgustinovich, S.V., Parshina, O.G., 
On perfect colorings of infinite multipath graphs}
\thanks{\copyright \ 2020 Lisitsyna M.A., Avgustinovich, S.V., Parshina O.G.}
\thanks{\rm This work was performed within the framework of the LABEX MILYON (ANR-10-LABX-0070) of Universit\'{e} de Lyon, within the program ``Investissements d'Avenir'' (ANR-11-IDEX-0007) operated by the French National Research Agency (ANR), and has been supported by RFBS grant 18-31-00009.}

\semrtop \vspace{1cm}
\maketitle {\small
\begin{quote}
\noindent{\sc Abstract.} A~coloring of vertices of a~given graph is called perfect if the color structure of each sphere of radius~$1$ in the graph depends only on the color of the sphere center.
Let $n$ be a~positive integer.
We consider a~lexicographic product of the infinite path
graph and a~graph $G$ that can be either the complete or
empty graph on~$n$ vertices.
We give a~complete description of perfect colorings with
an arbitrary number of colors of such graph products.
\medskip


\end{quote}
}

\section{Introduction}

Let $G$ be a~simple graph and $k$ a~positive integer.
A {\it perfect coloring} of the graph $G$ with the parameter matrix $M=(m_{ij})_{i,j=1}^k$ is a~map from the vertex set of the graph to the set of integers $\{1,2,3,\dots,k\}$ such that every vertex of color~$i$ is adjacent to exactly $m_{ij}$ vertices of color~$j$.

The concept behind the definition of the perfect coloring is quite natural. It arose and developed in connection with the problem of graph isomorphism recognition, and with the coding theory problems.
There are several equivalent notions indepen\-dently introduced in different contexts. For example, the notion of \textit{equitable partition} is used in works of C.~Godsil (ex.see~\cite{Godsil}). 
The notion of \textit{partition design} was introduced by P.~Camion, B.~Courteau, G.~Fournier, S.V.~Kanetkar in~\cite{CCFK} and is used in combinatorial design theory and in coding theory.

The problem of characterization of perfect colorings is an actual problem of coding theory, because the notion of perfect coloring is closely connected with many known codes such as perfect, completely regular and uniformly packed. For instance, a distance partition of a distance regular graph in accordance to a perfect code is a perfect coloring. It is worth mentioning that the coloring induced by a completely regular code introduced by P.~Delsarte~\cite{Delsarte} is perfect by definition.

The problem of existence of perfect codes in the $n$-dimensional hypercube graph has been attracting attention of mathematicians for more than half a century. 
Note that the best upper and lower bounds for a number of various 1-perfect codes differ essentially, what means that the complete description of them is far from being obtained. 
The perfect coloring with $k$ colors can be interpreted as a generalization of such codes in case of $k$-ary coding.

Sometimes the graph under consideration can be represented as the product of simpler graphs or graphs of lower dimensions. For example, the hypercube graph $E^{n}$ is the Cartesian product of graphs $E^{n-1}$ and $E$. 
Creation of the constructions that would help to obtain perfect colorings of graph products via colorings of their multipliers is a problem of interest in the areas of coding theory, algebraic combinatorics and graph theory.

Let us note that graphs products are interesting from the point of view of crystallography. Let $G^{*}$ be a product of graphs $G$ by $H$. The graph $G^{*}$ can be interpreted as a graph $G$ with vertices having the structure of kind $H$, i.e. every its vertex is a ``molecule of type $H$''. Perfect colorings of these graphs allow to model structures on crystals having several useful physical and chemical properties.

The lexicographic product of two graphs $G$ and $H$ is the~graph $G\cdot H$ such that its vertex set is the Cartesian product $V(G)\times V(H)$ and two vertices $(u_1, v_1)$ and $(u_2, v_2)$ are adjacent if and only if either $\{u_1,u_2\}\in E(G)$ or $u_1=u_2$ and $\{v_1,v_2\}\in E(H)$. 
The graph lexicographic product is also known as the graph composition~\cite{Harary}.

The {\it infinite path graph} is the graph whose set of vertices is the set of integers and two vertices $u$ and $v$ are adjacent if $|u-v|=1$.
Hereinafter, we denote the infinite path graph by~$C_\infty$ and, for a~transitive graph $G$, we call the graph $C_{\infty} \cdot G$ the {\it infinite $G$-times path}.

Let $n$ be a~positive integer.
By $K_n$ and $\overline{K_{n}}$ we denote the complete and empty graphs on~$n$ vertices.
In this paper we list the perfect colorings of $\overline{K_{n}}$- and $K_{n}$-times paths with an~arbitrary finite set of colors.
The local structure of the graphs $C_\infty\cdot \overline{K_3}$ and $C_\infty\cdot K_3$ is shown in Figure \ref{ris:Lisitsyna-figure}.

\begin{figure}[h!]
	\begin{center}
		\includegraphics[width=1.0\linewidth]{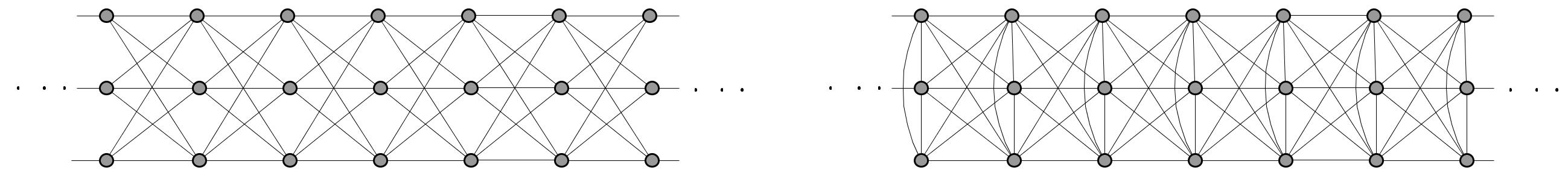}\par
		\caption{Local structure of $C_\infty\cdot \overline{K_3}$ (left) and $C_\infty\cdot K_3$ (right)}
		\label{ris:Lisitsyna-figure}
	\end{center}
\end{figure}

The graphs under consideration have an extensive structure, in other words, they contain $C_\infty$ as a~subgraph.
Perfect colorings of graphs with a~similar structure, such as
the infinite circulant graphs, infinite transitive grids,
infinite prism graph, were studied before.

The perfect colorings of the infinite prism graph with an~arbitrary finite number of colors are listed in~\cite{AvgLis2016}.
Several results on perfect colorings of circulant graphs are obtained by D.B.~Khoroshilova in~\cite{Khor11, Khor09}. She showed, in particular, that every perfect coloring listed in~\cite{Khor09} yields a perfect coloring of the $n$-dimensional infinite grid with the same parameter matrix.
Perfect 2-colorings for two families of infinite circulant graphs are
listed in~\cite{Par,Oddis}.
The complete description of  perfect colorings with an~arbitrary finite number of colors is obtained for the infinite circulant graphs with distances~$1$ and~$2$ in~\cite{LisPar2017}. 
Let us note that perfect colorings of circulant graphs can be used in mathematical optimization~\cite{table}.

A~coloring of a~graph is called {\it perfect of radius $r$} with the parameter matrix $M=(m_{ij})$ if for every vertex $x$ of color~$i$ the number of vertices of color~$j$ in the sphere of radius~$r$ with center~$x$ is equal to~$m_{ij}$.

First results on perfect colorings of the infinite rectangular grid graph $G(Z^2)$ were obtained by M.~Axenovich~\cite{Axenovich}.
She listed all admissible parameter matrices of perfect $2$-colorings of radius~$1$ for this graph and established several necessary conditions for a~matrix to be admissible for the graph in the case $r \geq 2$.
Parameters and properties of perfect colorings of $G(Z^2)$ were studied by S.A.~Puzynina in her thesis.
In~\cite{Puz_GZ2_Periodicity, PuzAvg2008} she showed that all perfect colorings of the infinite rectangular grid of radius $r > 1$ are periodic and proved
their pre-periodicity in the case $r = 1$.
A~technique of equivalent colors merging is proposed in~\cite{Puz_GZ2_Periodicity}; we will use this technique to prove the main result.
All admissible parameter matrices of order~$3$ for the graph $G(Z^2)$ 
were described in~\cite{Puz_GZ2_3-Colors}.
Perfect colorings with up to $9$~colors of this graph are listed by D.S.~Krotov in~\cite{KrotovG(Z2)}.

A~perfect coloring is called {\it distance regular} if its
parameter matrix can be reduced to the tridiagonal form.
The parameters of all distance regular colorings of the
infinite rectangular grid were listed by S.V.~Avgustinovich,
A.Yu.~Vasil'eva, and I.V.~Sergeeva in~\cite{AvgVasSerg}.

The pre-periodicity of perfect colorings of the hexagonal and
triangular grids was proven by S.A.~Puzynina in~\cite{Puzynina2011}.
For the infinite triangular grid, 
the distance regular colorings were listed by A.Yu.~Vasil'eva 
in~\cite{Vasil'eva2014}; for the hexagonal grid, 
they were later studied 
by S.V.~Avgustinovich, D.S.~Krotov, and A.Yu.~Vasil'eva~\cite{AvgKrotovVas}.

\section{Disjunctive perfect colorings of the graph
	lexicographic product}

Let $G$ and $H$ be simple graphs, where $G$ may be an
infinite graph, and $G \cdot H$ their lexicographic
product.

Let $k$ be a~positive integer.
The elements of the finite set
$I = \{1,2, \ldots ,k\}$ are called the {\it colors}.
Let $\psi : V(G) \rightarrow I$ be a~perfect coloring of
the graph $G$ and
$\Phi = \{\phi_1, \phi_2, \ldots \phi_k\}$ a~set of
perfect colorings of the graph $H$ with colors from sets
$J_1, J_2, \ldots J_k$ respectively, where
$J_p \cap J_q = \varnothing$ if $p \neq q$. 
We define the following coloring for the graph $G \cdot H$:
\begin{center}
	$\psi \cdot \Phi : V(G) \times V(H) \rightarrow J_1 \cup J_2 \cup \ldots \cup J_k$;
	\\
	$\psi \cdot \Phi(v_1,v_2)=\phi_{\psi(v_1)}(v_2)$.
\end{center}
Such a~structure on the graph $G \cdot H$ is called 
a~{\it disjunctive coloring}.
The formula in the definition  
reflects the fact that the perfect coloring $\phi_i$ of $H$ with colors from $J_i$ corresponds to the
color $i$ in the perfect coloring of $G$.


\begin{lemma} 
	A~disjunctive coloring of the graph $G \cdot H$ is perfect.
	\label{DisjunctiveColoringsLemma} 
\end{lemma}

\begin{proof} 
	Consider two vertices $(u_1,u_2)$ and
	$(w_1,w_2)$ colored with the same color in the structure
	$\psi \cdot \Phi(v_1,v_2)$:
	
	\begin{center}
		$\psi \cdot \Phi(u_1,u_2) = \psi \cdot \Phi(w_1,w_2) 
		\Rightarrow \phi_{\psi(u_1)}(u_2) = \phi_{\psi(w_1)}(w_2)$.
	\end{center}
	
	By the definition of the disjunctive coloring, 
	$\psi(u_1) = \psi(w_1)$.
	The vertices $u_1$ and $w_1$ are colored with the same
	color in the perfect coloring of~$G$; hence, the color
	structures of their neighborhoods coincide.
	Therefore, the copies of the graph $H$ neighboring $u_1$ and
	$w_1$ have equal multisets of colors.
	The colors of adjacent vertices from the $H$-copies corresponding
	to $u_1$ and $w_1$ are the same, since the colorings
	$\phi_{\psi(u_1)}$ and $\phi_{\psi(w_1)}$ are perfect.
	Thus, the neighborhoods of the vertices $(u_1,u_2)$ and $(w_1,w_2)$
	have the same color structure.
	Consequently, the coloring $\psi \cdot \Phi(v_1,v_2)$ is
	perfect.
\end{proof}

Herein, we consider the case when
$G = C_{\infty}$ and $H = \overline{K_{n}}$ or $H = K_{n}$.

We denote by $V_i$ the~copy of $\overline{K_{n}}$
with number~$i$ in the corresponding multipath and call it $i$-th {\it block}. The vertices of each block are
enumerated with the integers from $1$ to~$n$.
The $j$-th vertex of the $i$-th block in
$\overline{K_{n}}$-times path is denoted by $v_{ij}$.
We use the same enumeration when considering
the graph $C_{\infty} \cdot K_{n}$.

Note that any coloring of the empty or complete graph is perfect;
therefore, it can be used to construct a~disjunctive
coloring of $C_{\infty} \cdot \overline{K_{n}}$ and
$C_{\infty} \cdot K_{n}$ respectively.

\section{Equivalent colors in a~perfect coloring}
\label{sect3}

Consider a~finite regular graph $G=(V,E)$ and a~perfect
coloring $\phi : V \rightarrow I$ with the parameter matrix~$M$.
Two colors $i$ and $j$ in the perfect coloring
$\phi$ are called {\it equivalent} ($i \sim j$) if the coloring
obtained after their identification is perfect.
Note that the rows of the parameter matrix
corresponding to the colors $i$ and $j$ coincide up to the
elements of the columns $i$ and~$j$.
This property of the parameter matrix is equivalent to the
definition of equivalent colors.

\begin{lemma} 
	The relation $``\sim"$ defined above is an
	equivalence relation.
	Moreover, the coloring obtained by identifying colors in
	equivalent classes is perfect.
	\label{EquivalenceRelationLemma} 
\end{lemma}

\begin{proof} 
	Reflexivity and symmetry of the relation are obvious.
	To show transitivity, consider the colors $a$, $b$,
	and $c$ of a~perfect coloring $\phi$ such that $a \sim b$
	and $b \sim c$.
	Without loss of generality, we may suppose that $a = 1$,
	$b = 2$, and $c = 3$.
	The fragment of the parameter matrix $M$  
	corresponding to these colors has the form 
	\begin{center}
		$\begin{pmatrix} & x & p & \ldots \\ y & & p & \ldots \\ y & q & & \ldots \\ \vdots & \vdots & \vdots & \ddots \end{pmatrix}$;
	\end{center}
	the non-identified elements of the first
	three rows of $M$ coincide by the definition of~$``\sim"$.
	
	Note that $m_{ij}=0 \Leftrightarrow m_{ji}=0$ 
	for any two colors $i$ and~$j$.
	This, in particular, means that if at least one of the
	numbers $x$, $y$, $p$, and $q$ equals zero, then so do the others.
	In this case $1\sim 3$.
	
	Consider the case when none of the elements 
	$x$, $y$, $p$, and $q$ is equal to zero.
	We denote the number of vertices colored with $i$ by $N_i$.
	Consider three subgraphs of~$G$ induced by the sets of
	vertices colored with the colors $1$ and~$2$, $2$ and~$3$, $1$ and~$3$.
	Each of the induced subgraphs is a~biregular bipartite
	graph, and the following relations hold:
	\begin{center}
		$\frac{N_1}{N_2}=\frac{y}{x};\quad \frac{N_2}{N_3}=\frac{q}{p};\quad \frac{N_3}{N_1}=\frac{p}{y}$,
	\end{center}
	what leads to the equalities:
	\begin{center}
		$\frac{N_1}{N_2} \cdot \frac{N_2}{N_3} \cdot \frac{N_3}{N_1}
		=\frac{y}{x} \cdot \frac{q}{p} \cdot \frac{p}{y} \Rightarrow \frac{q}{x}=1 \Rightarrow x=q$.
	\end{center}
	The latter means that the colors~$1$ and~$3$ are equivalent,
	which proves transitivity of the relation.
	Thus $``\sim"$ is an equalence relation.
	
	The set $I$ can be split into disjoint equivalence
	classes by the relation $\sim$.
	Let us show that the coloring obtained by identifying
	colors in equivalent classes is perfect.
	
	
	Such a~coloring, denoted by $\hat\phi$, is obtained in the following way:
	every two vertices colored in equivalent colors in~$\phi$
	get the same color in~$\hat\phi$.
	The color corresponding to the equivalence class $[j]$ is
	denoted by $j^{*}$.
	The rows of the parameter matrix corresponding to the
	elements $x$ and $y$ from the class $[j]$ coincide up to
	the elements of the columns $x$ and~$y$.
	By regularity of~$G$, the sums $m_{xx}+m_{yx}$ and
	$m_{xy}+m_{yy}$ are equal, where the sums stand for the
	number of vertices of color $j^{*}$ adjacent to a~vertex
	colored in~$j^*$.
	Thus, the number of vertices of each color in every unit
	sphere with the center of color $j^*$ is the same
	and the coloring $\hat\phi$ is perfect.
\end{proof}

The identifying of colors in each class of
equivalence from the factor-set $I \slash \sim$ 
is referred to as the \textit{gluing operation}.
The coloring $\hat\phi$ obtained by gluing the colors of some perfect
coloring $\phi$ is called a~\textit{reduced} coloring.
Let $\phi$ be a~reduced coloring.
Every perfect coloring $\psi$ such that $\phi = \hat{\psi}$ 
is called the \textit{splitting} of~$\phi$.
Observe that a~reduced coloring $\phi$ can have 
several different splittings.
Thus, to enumerate all perfect colorings of the graph~$G$,
it suffices, first, to obtain a~complete description of its 
reduced colorings and, then, to consider all admissible splittings 
of the latter.

\section{Reduced colorings of $\overline{K_{n}}$- and $K_{n}$-times paths}

Any perfect coloring of the graph $C_{\infty} \cdot \overline{K_{n}}$ (or $C_{\infty} \cdot K_{n}$) is described by the sequence of color sets of it blocks up to equivalence. It is easy to be shown (using Dirichlet principle) that any such sequence is periodic, as if parameters of the perfect coloring are known and coloring of two neighbour blocks is fixed, then the whole coloring is uniquelly restored. Length of period of perfect coloring in this case is equal to number of blocks in period of such sequence of color sets.

To describe a perfect coloring of such graph, it suffices 
to indicate its least period,
which means that the perfect colorings of $\overline{K_{n}}$- and $K_{n}$-times paths
are the homomorphic inverse images of the perfect colorings
of the corresponding finite graphs. The latter allows us to apply methods and results of
Section~\ref{sect3} to infinite multipath graphs under consideration.

A~perfect coloring of the graph $C_{\infty} \cdot \overline{K_{n}}$ ( or $C_{\infty} \cdot K_{n}$)
is called {\it block-monochrome\/} 
if the vertices have the same color in each block.
The period of the block-monochrome perfect coloring of the graph
$C_{\infty} \cdot \overline{K_{n}}$ (or $C_{\infty} \cdot K_{n}$) will be denoted by  
a~row in square brackets whose elements 
are the colors of the blocks.

\begin{lemma} 
	The reduced colorings of the graphs
	$C_{\infty} \cdot \overline{K_{n}}$ and
	$C_{\infty} \cdot K_{n}$ are block-monochrome.
	\label{MonochromeBlocksColoringLemma} 
\end{lemma}

\begin{proof} 
	Consider the perfect colorings of the graphs
	from $\overline{K_{n}}$- and $ K_{n}$-times paths.
	The neighborhoods of the vertices of a~$\overline{K_{n}}$-copy in 
	$C_{\infty} \cdot \overline{K_{n}}$ coincide;
	therefore, the colors assigned to such vertices are equivalent.
	The neighborhoods of two vertices from the same block of 
	$C_{\infty} \cdot K_{n}$ differ in exactly these two vertices.
	Thus, the colors corresponding to such vertices are also equivalent.
	
	Hence, gluing the perfect colorings of such graphs 
	leads to the vertices of each copy to be colored with one color, 
	i.e., a~block-monochrome perfect coloring is obtained.
\end{proof}

There is a~one-to-one correspondence between
the perfect colorings of an~infinite path graph 
and the block-monochrome perfect colorings of $C_{\infty} \cdot \overline{K_{n}}$ and $C_{\infty} \cdot K_{n}$; 
the latter are obtained by $\overline{K_{n}}$-times or $K_{n}$-times copying of the former.

The perfect colorings of infinite path graph, 
in their turn, are well studied.
We recall the terminology and well-known facts 
that can be found in~\cite{AvgLis2016}.

Every perfect coloring of $C_{\infty}$ is periodic.
We also denote its period by a~row in square brackets.
The same notation is chosen due to the fact that the~infinite
path graph is a~particular case of the graph $C_{\infty} \cdot \overline{K_{n}}$.
The colorings of the~infinite path graph with the periods
$S_{11}(k)=[k{-}1\;k{-}2\;\ldots\;1\;0\;1\;\ldots\;k{-}3\;k{-}2]$,
$S_{12}(k)=[k{-}1\;k{-}2\;\ldots\;1\;0\;1\;\ldots\;k{-}2\;k{-}1]$,
and $S_{22}(k)=[k{-}1\;k{-}2\;\ldots\;1\;0\;0\;$ \-$1\;\ldots\;k{-}2\;k{-}1]$
are called the \textit{mirror colorings of types~$(1, 1)$, $(1, 2)$ and~$(2,2)$};
with the periods
$S(k)=[0\;1\;2\;\ldots\;k{-}2\;k{-}1]$, they are the \textit{cyclic} colorings.
The type of a~mirror coloring is determined by the number of vertices 
of the first and last colors ($0$ and $k{-}1$) in the period.

Henceforth, by the cyclic and mirror colorings
of $\overline{K_{n}}$- and $K_{n}$-times paths we understand 
the block-monochrome colorings of these graphs
corresponding to the cyclic and mirror colorings of~$C_{\infty}$.

The following lemma describes all perfect colorings
of the~infinite path graph \cite{AvgLis2016}:

\begin{lemma} 
	The perfect colorings of the graph $C_{\infty}$ are exhausted
	by the following four infinite series: 
	three series of mirror 
	and one series of cyclic colorings.
	\label{PerfectColoringsOfInfinitePath} 
\end{lemma}

\begin{corollary} 
	The reduced colorings of the graphs
	$C_{\infty} \cdot \overline{K_{n}}$ and
	$C_{\infty} \cdot K_{n}$ are exhausted by four infinite series
	corresponding to the perfect colorings of the graph $C_{\infty}$.
	\label{MonochromeBlockColoringsOfMultiPathCorollary} 
\end{corollary}

Thus, we obtained a~complete description of the reduced perfect
colorings of the infinite $\overline{K_{n}}$- and $K_{n}$-times path graphs.

Let $V_i$ be a~block of one of the multipath graphs under consideration,
$\phi$ the reduced coloring of that multipath graph, 
and $\psi$ its splitting.
Thus, the multiset of colors corresponding to the vertices of this block
in the coloring $\psi$ is $\psi(V_{i})$.

\begin{lemma} 
	Let $V_i$ and $V_j$ be the blocks of $C_{\infty} \cdot \overline{K_{n}}$ or
	$C_{\infty} \cdot K_{n}$ the colors of which in the reduced coloring
	$\phi$ do not coincide and are equal to $a$ and $b$ respectively.
	Then $\psi(V_{i}) \cap \psi(V_{j}) = \varnothing$
	for every splitting $\psi$ of $\phi$. 
	\label{EmptyIntersectionOfDifferentColoredBlocksSplit_lemma}
\end{lemma}

\begin{proof} 
	Otherwise, by applying the gluing operation to the perfect coloring $\psi$, 
	we obtain $a=b$.
\end{proof}

Note that if there are blocks $V_{p}$ and $V_{q}$ 
in the graph $C_{\infty} \cdot G$ such that 
$\psi(V_{p}) \neq \psi(V_{q})$ and
$\psi(V_{p}) \cap \psi(V_{q}) \neq \varnothing$, then
the coloring $\psi$ is not disjunctive.

Thus, to complete the characterization of the perfect colorings
of the $\overline{K_{n}}$- and $K_{n}$-times path graphs, 
it remains to describe all disjunctive and non-disjunctive splittings  
of the reduced colorings of these graphs.

\section{Perfect colorings of the~$\overline{K_{n}}$-times path}

The graph $C_{\infty} \cdot \overline{K_{n}}$ being bipartite
is one of its important structural properties.
This allows us to use the ideas on colorings of bipartite
graphs presented in~\cite{AvgLis2016}.
Let us give some necessary definitions.

Let $G(V_{1},V_{2})$ be a~bipartite graph with the parts
$V_{1}$ and $V_{2}$.
A~coloring of one of the graph parts is a~{\it semicoloring} of~$G$.
A~semicoloring is called \textit{admissible} if it is a~part
of the~perfect coloring of~$G$. If a~semicoloring belongs to the~reduced coloring of~$G$ it is \textit{reduced} semicoloring.
Two admissible semicolorings of a~graph are {\it conjugate},
if they complement each other to make a~perfect coloring of the graph.
 
Admissible semicolorings of $C_{\infty} \cdot \overline{K_{n}}$ are periodic, because perfect colorings of the whole graph are periodic. Length of period of admissible semicoloring is equal to number of blocks in such period. 

A~perfect coloring of a~bipartite graph is \textit{bipartite}
if the color sets of its semicolo\-rings are disjoint;
otherwise, the coloring is \textit{non-bipartite}.
Note that the color sets of semicolorings coincide in non-bipartite case if $G$ is connected.

In non-disjunctive perfect colorings, the number of vertices
of the given color in different blocks can be different.
Denote by $N_{j}(i)$ 
the number of $j$-colored vertices in the block $V_{i}$
of~$C_{\infty} \cdot \overline{K_{n}}$ or
$C_{\infty} \cdot K_{n}$.

Let us describe a construction for the~$\overline{K_{n}}$-times path. We consider a coloring $\psi$ of $C_{\infty} \cdot \overline{K_{n}}$ with period of length 4. If equality  $N_{j}(i-1) + N_{j}(i+1)=N_{j}(i) + N_{j}(i+2)$ holds for every color $j$ of $\psi$ and every $i$ then semicolorings of $\psi$ are called \textit{matched}. It is easy to be shown that the validaty of latter condition for all colors of $\psi$  implies its perfectness.

To make the structure of the further arguments clear, let us
formulate the main result of this section:

\begin{theorem} 
	The perfect colorings of the graph
	$C_{\infty} \cdot \overline{K_{n}}$ 
	are exhausted by the following list:
	
	$1$~Disjunctive perfect colorings; 
	
    $2$~Non-disjunctive bipartite  colorings obtained by conjugation
	of $2$-periodic semicolo\-rings with disjoint sets of colors;
	
	$3$~Non-disjunctive non-bipartite colorings obtained by conjugation
	of two matched $2$-periodic semicolorings.
	\label{EmptyGraphTimesPathTheorem}
\end{theorem}

We start the study of the perfect colorings of the~$\overline{K_{n}}$-times 
path graph with a~description of its admissible reduced
semicolorings.

\begin{lemma} 
	The admissible reduced semicolorings of the graph
	$C_{\infty} \cdot \overline{K_{n}}$ are exhausted 
	by the following four infinite series: 
	three series of mirror 
	and one series of cyclic semicolorings.
	\label{FeasibleHalfcoloringsOfEmptyGraphTimesPath_lemma}
\end{lemma}

\begin{proof} 
	Partitioning any reduced perfect
	coloring of the graph $C_{\infty} \cdot \overline{K_{n}}$
	(see Corollary \ref{MonochromeBlockColoringsOfMultiPathCorollary}) into semicolorings produces two color sequences. Each of them belongs to cyclic or one of three mirror series. For example, the reduced coloring of $\overline{K_{n}}$-times path with period  $S_{22}(3)=[2\;1\;0\;0\;1\;2]$ is obtained by conjugation of two cyclic semicolorings -- $[0\;1\;2]$ and $[0\;2\;1]$.
\end{proof}

In Lemma~\ref{BlockMonochromeColoringsWithNonDisjunctiveSplit_lemma}
and Corollary~\ref{BlockMonochromeColoringsWithNonDisjunctiveSplit_corollary}
we characterize the reduced colorings of the~$\overline{K_{n}}$-times
path graph admitting a~non-disjunctive splitting.

\begin{lemma} 
	If a~perfect coloring $\psi$ of the graph
	$C_{\infty} \cdot \overline{K_{n}}$ is obtained by 
	non-disjunctive splitting of a~reduced coloring $\phi$, 
	then $\phi$ is a~conjugation of one-color semicolorings 
	(either bipartite or non-bipartite).
	\label{BlockMonochromeColoringsWithNonDisjunctiveSplit_lemma}
\end{lemma}

\begin{proof} 
	Let $V_i$ and $V_j$ be two $a$-colored copies
	of the~empty graph admitting a~non-disjunctive splitting, i.e.,
	$\psi(V_{i}) \neq \psi(V_{j})$.
	Consider the following two cases: 
	the blocks $V_i$ and $V_j$ 
	belong to one part of the graph
	$C_{\infty} \cdot \overline{K_{n}}$ in the first case
	and to different parts in the second.
	
	Study the first case.
	Assume that the reduced semicoloring of the part which 
	$V_i$ and $V_j$ belong to contains more than one color.
	By Lemma~\ref{FeasibleHalfcoloringsOfEmptyGraphTimesPath_lemma},
	at least one of the blocks $V_{i-2}$ or $V_{i+2}$ in
	the coloring $\phi$ is colored with a~color~$b$ 
	different from~$a$.
	Without loss of generality, let $V_{i+2}$ be such a~block.
	Let the block $V_{i+1}$ be colored with a~color~$x$, 
	which can coincide with~$a$ or~$b$.
	By Corollary~\ref{MonochromeBlockColoringsOfMultiPathCorollary},
	the right or left neighbor of the copy of $V_j$ is also of color~$x$.
	For definiteness, let $V_{j+1}$ be such a~copy; 
	then $V_{j+2}$ is colored with~$b$.
	
	The color sets of the neighborhoods of the $x$-colored vertices in
	the coloring $\psi$ coincide;
	therefore, 
	\begin{center}
		$\psi(V_{i}) \cup \psi(V_{i+2}) = \psi(V_{j}) \cup \psi(V_{j+2})$.
	\end{center} 
	By Lemma~\ref{EmptyIntersectionOfDifferentColoredBlocksSplit_lemma}, 
	the latter equality is valid only in the case when
	$\psi(V_{i}) = \psi(V_{j})$ and $\psi(V_{i+2})= \psi(V_{j+2})$;
	a~contradiction.
	Hence, the part which $V_i$ and $V_j$ belong to
	is monochrome colored in~$\phi$.
	
	Describe all reduced perfect colorings 
	of the graph under consideration 
	to which belong that semicoloring.
	It is easy to see that 
	these are the colorings $S(1)$, $S(2)$, and $S_{11}(3)$.
	Note that $S_{11}(3)$ cannot be obtained by gluing the colors 
	of the perfect coloring: the corresponding $4$-periodic perfect colorings
	get glued immediately in~$S(2)$.
	
	For the second case, the proof is similar, 
	with the only difference that we suppose that at least one
	semicoloring in $\phi$ is not one-colored;
	a~contradiction.
	
	Thus, the reduced coloring admitting a~non-disjunctive splitting
	is a~conjugation of one-colored semicolorings 
	(bipartite or non-bipartite).
\end{proof}

\begin{corollary} 
	The only reduced colorings of the graph
	$C_{\infty} \cdot \overline{K_{n}}$ 
	admitting non-disjunctive splittings
	are $S(1)$ (non-bipartite case) and $S(2)$ (bipartite case).
	\label{BlockMonochromeColoringsWithNonDisjunctiveSplit_corollary}
\end{corollary}

All non-disjunctive colorings of the~$\overline{K_{n}}$-times path graph
are described in Lemma~\ref{NonDisjunctiveColoringsOfEmptyGraphTimesPath_lemma}.

\begin{lemma} 
	The non-disjunctive perfect colorings of the graph
	$C_{\infty} \cdot \overline{K_{n}}$ 
	are exhausted by the following list: 
	
	$1$~Non-disjunctive bipartite colorings obtained by conjugation
	of arbitrary $2$-periodic semicolorings; 
	
	$2$~Non-disjunctive non-bipartite colorings obtained by conjugation
	of two matched $2$-periodic semicolorings.
	\label{NonDisjunctiveColoringsOfEmptyGraphTimesPath_lemma}
\end{lemma}

\begin{proof} 
	By Corollary~\ref{BlockMonochromeColoringsWithNonDisjunctiveSplit_corollary},
	we shall describe the non-disjunctive splittings 
	of the colorings $S(1)$ and~$S(2)$.
	Consider the sequence of blocks $V_{i-1}$,
	$V_{i}$, $V_{i+1}$, $V_{i+2}$, and $V_{i+3}$.
	Due to the fact that $V_{i+1}$ and $V_{i+3}$ 
	have the same color in $S(1)$ and $S(2)$ 
	and the vertices of $V_{i+2}$
	are in the $1$-neighborhoods of both $V_{i+1}$ and $V_{i+3}$,
	we obtain $\phi(V_{i})=\phi(V_{i+4})$ for every splitting
	of such block-monochrome colorings.
	Since $i$ may be arbitrary, 
	the perfect coloring $\phi(v)$ is $4$-periodic
	and its semicolorings have the period of length~$2$.
	
	Note that the result of conjugation of arbitrary $2$-periodic
	semicolorings of $C_{\infty} \cdot \overline{K_{n}}$
	with disjoint sets of colors is a~perfect coloring.
	The set of non-disjunctive bipartite perfect colorings 
	of the~$\overline{K_{n}}$-times path graph consists of such conjugations.
	
	In a~non-disjunctive non-bipartite perfect coloring $\psi$
	of $C_{\infty} \cdot \overline{K_{n}}$,
	there is a~pair of the same-colored vertices 
	in adjacent blocks $V_{i}$ and~$V_{i+1}$.
	The sets of colors of their neighborhoods coincide; 
	therefore, $N_{j}(i-1) + N_{j}(i+1)=N_{j}(i) + N_{j}(i+2)$ 
	for each color~$j$ of~$\psi$. Thus, perfect coloring $\psi$ is obtained by conjugation of two matched semicolorings.  
%
	
	The characterization of non-disjunctive perfect colorings of 
	a~$\overline{K_{n}}$-times path graph is completed.
\end{proof}

The assertion of Theorem~\ref{EmptyGraphTimesPathTheorem}
follows from Lemma~\ref{NonDisjunctiveColoringsOfEmptyGraphTimesPath_lemma}.

\section{Perfect colorings of the~$K_{n}$-times path graph}

\begin{theorem} 
	The perfect colorings of the graph
	$C_{\infty} \cdot K_{n}$ are exhausted by the following list: 
	
	$1$~Disjunctive perfect colorings; 
	
	$2$~Non-disjunctive $3$-periodic colorings.
	\label{CompleteGraphTimesPathTheorem} 
\end{theorem}

\begin{proof}
	The proof can be reduced to description of non-disjunctive
	splittings of the reduced perfect colorings of 
	the~$K_{n}$-times path graph.
	Let $\phi$ be a~reduced coloring of the graph under consideration 
	and $\psi$ one of its non-disjunctive splittings.
	Consider the following two cases: 
	with the coloring $\phi$ either being one-colored 
	or consisting of more than one color.
	
	Studying the first case,
	we consider the blocks following one another,
	starting with~$V_{i}$.
	Show that $\psi(V_i)=\psi(V_{i+3}) (*)$.
	If there exists a~pair of
	vertices of the same color in $V_{i+1}$ and $V_{i+2}$, 
	then equality $(*)$ holds.
	Suppose that there are no same-colored vertices in these blocks 
	and $\psi(v_{(i+1)j}) =p$ and $\psi(v_{(i+2)l})=q$ for some $j$
	and~$l$.
	Since $p$ and $q$ are equivalent, the number of vertices of color $s$
	$(s \neq p$, $s \neq q)$ adjacent to $v_{(i+1)j}$ is equal to
	the number of such neighbors of~$v_{(i+2)l}$:
	\begin{center}
		$N_{s}(i) + N_{s}(i+1) + N_{s}(i+2)
		=N_{s}(i+1) + N_{s}(i+2) + N_{s}(i+3) \Rightarrow N_{s}(i)=N_{s}(i+3)$.
	\end{center}
	
	If there is a~vertex of color~$t$
	different from $p$ and~$q$ in $V_{i+1}$ or $V_{i+2}$, 
	then $t \sim p$ and $t \sim q$.
	For these colors, by writing down an equality similar to the latter,  
	we obtain 
	$N_{q}(i)=N_{q}(i+3)$ and $N_{p}(i)=N_{p}(i+3)$;
	consequently, $\psi(V_i)=\psi(V_{i+3})$.
	
	The lack of vertices of color $t$ means that the blocks $V_{i+1}$
	and $V_{i+2}$ are monochrome-colored with the colors $p$ and~$q$
	respectively.
	Show that all admissible extensions of such a~fragment are $3$-periodic
	perfect colorings.
	
	If there is a~color~$r$ in the set $\psi(V_{i+3})$ such that
	$r \neq p$ and $r \neq q$, 
	then by equivalence of $q$ and~$r$
	the block $V_{i+4}$ is monochrome-colored with color~$p$.
	The multiset $\psi(V_{i+5})$ consists only of the elements $q$, 
	since $p \sim r$.
	Such a~fragment is uniquely extended to a~$3$-periodic perfect
	coloring.
	
	Let only $p$ and~$q$ be the elements of the set $\psi(V_{i+3})$;
	hence, the coloring $\psi$ is two-colored.
	In the case $q \in \phi (V_{i+3})$, the color set of the 
	neighborhoods of vertices of color~$q$ is defined uniquely.
	This allows us to extend the coloring $\phi(v)$ to the right 
	by a~$p$-colored block. The admissible extensions of this fragment are exhausted by $3$-periodic
	non-disjunctive and two disjunctive colorings $S_{12}(2)$ and $S_{22}(2)$.
	
	In the case of the~$p$-monochrome block $V_{i+3}$, the perfect
	extensions of such a~structure are exhausted by the disjunctive colorings
	$S_{12}(2)$ and $S(2)$.
	The characte\-rization of the non-disjunctive splittings of a~one-colored 
	reduced coloring is thus completed.
	
	Consider the case of a~reduced coloring $\phi$ with two or more colors.
	Let $V_i$ and $V_j$ be copies of a~complete graph admitting 
	non-disjunctive splitting;
	i.e., the elements of such blocks in $\phi$ are colored with the same color~$a$, 
	but $\psi(V_i) \neq \psi(V_j)$.
	
	At least one of the blocks adjacent to $V_i$ is colored differently 
	in the reduced coloring, with a~color  $b$ ($b \neq a$).
	This is also true for the neighbors of the copy of~$V_j$.
	Without loss of generality, we assume that 
	the blocks $V_{i+1}$ and $V_{j+1}$ are $b$-colored.
	Since
	$\psi(V_i) \neq \psi(V_j)$ and the color sets of the neighborhoods 
	of vertices from $V_{i}$ and $V_{j}$ in the coloring
	$\psi(v)$ coincide, we have
	$\psi(V_{i-1}) \cap \psi(V_{j}) \neq \varnothing$ and
	$\psi(V_{i}) \cap \psi(V_{j-1}) \neq \varnothing$.
	Consequently, the color~$a$ corresponds to the elements 
	of the copies $V_{i-1}$ and $V_{j-1}$ in $\phi$.
	Arguing similarly for $V_{i+1}$ and $V_{j+1}$, we find that
	the vertices of $V_{i+2}$ and $V_{j+2}$ are also $a$-colored
	in this coloring.
	Thus, the period of the reduced coloring $\phi$ has the form $[aba]$.
	
	Prove that every non-disjunctive splitting $\psi$ of such
	a~reduced coloring is $3$-periodic.
	Consider the sequence of blocks 
	$V_{i}$, $V_{i+1}$, $V_{i+2}$, and $V_{i+3}$.
	Let their colors in the coloring $\phi$ be equal to $a$, $a$, $b$, and $a$ 
	respectively and $\psi(V_i) \neq \psi(V_{i+1})$.
	Show that $\psi(V_i) = \psi(V_{i+3})$.
	
	Suppose that this is not true.
	Hence, there is a~color~$c$ such that the numbers 
	of vertices of this color in the blocks 
	$V_{i}, V_{i+1}$, and $V_{i+3}$ 
	are equal to $x$, $y$, and $z$ respectively, 
	while $x \neq z$.
	Without loss of generality, we can assume that $z < x$, i.e.,
	$x-z > 0$.
	The number of neighbors of color~$c$ in the neighborhood of each vertex
	is thus defined.
	For the vertices to which the color~$a$ corresponds 
	in the reduced coloring, this number is equal to $x+y$, 
	while for the elements of the $b$-colored copies it is $y+z$.
	Calculating the number of vertices of color~$c$ in the blocks
	$V_{i+4}$, $V_{i+5}$, $V_{i+6}$, etc., 
	we obtain $N_{c}(i+3p)=z-(p-1)(x-z)$.
	This implies that $N_{c}(i+3p)$ is monotone decreasing 
	with growth of $p$, which contradicts the infinity 
	of the graph under consideration;
	therefore, our assumption is false and $x=z$.
	Hence, we obtain the $3$-periodicity of the coloring $\psi$.
	
	Consequently, all non-disjunctive perfect colorings of a~$K_n$-times
	path graph have the period of length~$3$.
\end{proof}

Thus, the set of the perfect colorings of the graph $C_{\infty} \cdot K_{n}$
consists of two infinite series: disjunctive colorings and
non-disjunctive $3$-periodic colorings.

\section*{Conclusion} 

Creating the constructions
that make it possible to obtain 
perfect colorings of different types of graph products
from the perfect colorings of their factors
is an~important 
problem of graph theory. In particular, the $n$-dimensional binary cube 
$E^n$ can be presented as the product of hypercubes of smaller dimension.
The complete description of its perfect colorings is not known yet even in the case of two colors.

In this article, we study the lexicographic product of graphs. We show that  the set of perfect colorings of the graph $G \cdot H$ splits up into two subsets -- disjunctive and non-disjunctive colorings. Exploring constructions of non-disjunctive colorings for different pairs $(G,H)$ is a natural and interesting problem.

The simplest graph of such a~type is the lexicographic
product of the infinite path graph 
and an~arbitrary transitive graph $G$, i.e.
the~$G$-times path graph.
We described all perfect colorings of 
the $\overline{K_{n}}$- and $K_{n}$-times path graphs with an~arbitrary 
finite number of colors.

The multipath graphs can be viewed as extensions of the~infinite path graph and, consequently, the structures built on them may find applications in the group theory and crystallography.

\end{document}